\documentclass[12pt,reqno]{amsart}
\usepackage{amssymb,amsmath,amsthm,hyperref}
\newtheorem{theorem}{Theorem}

\newtheorem{proposition}[theorem]{Proposition}

\theoremstyle{remark}

\theoremstyle{definition}

\numberwithin{theorem}{section} \numberwithin{equation}{section}
\numberwithin{example}{section}

\title{Eulerian series as Modular forms revisited}

\author{Eric T. Mortenson}

\begin{document}

\date{24 January 2014}



\begin{abstract}
Recently, Bringmann, Ono, and Rhoades employed harmonic weak Maass forms to prove results on Eulerian series as modular forms.   By changing the setting to Appell--Lerch sums, we shorten the proof of one of their main theorems. In addition we discuss connections to recent work of Kang.
\end{abstract}

\address{Max-Planck-Institut f\"ur Mathematik, Vitvatsgasse 7, 53111 Bonn, Germany}
\email{etmortenson@gmail.com}
\maketitle
\setcounter{section}{-1}

\section{Definitions and Introduction}

 Let $q$ be a complex number, $0<|q|<1$, and define $\mathbb{C}^*:=\mathbb{C}-\{0\}$.  We recall
\begin{gather}
(x)_n=(x;q)_n:=\prod_{i=0}^{n-1}(1-q^ix), \ \ (x)_{\infty}=(x;q)_{\infty}:=\prod_{i\ge 0}(1-q^ix),\notag\\
 j(x;q):=(x)_{\infty}(q/x)_{\infty}(q)_{\infty}=\sum_{n=-\infty}^{\infty}(-1)^nq^{\binom{n}{2}}x^n,\label{equation:theta-series}
\end{gather}
where in the last line the equivalence of product and sum follows from Jacobi's triple product identity.  The following are special cases of the above definition.  Let $a$ and $m$ be integers with $m$ positive.  Define
\begin{gather*}
J_{a,m}:=j(q^a;q^m), \ \ \overline{J}_{a,m}:=j(-q^a;q^m), \ {\text{and }}J_m:=J_{m,3m}.
\end{gather*}

We will use the following definition of an Appell-Lerch sum.  Using the notation of \cite{HM}:
\begin{equation}
m(x,q,z):=\frac{1}{j(z;q)}\sum_{r=-\infty}^{\infty}\frac{(-1)^rq^{\binom{r}{2}}z^r}{1-q^{r-1}xz}.\label{equation:mxqz-def}
\end{equation}
Appell--Lerch sums are useful in studying $q$-hypergeometric series \cite{HM, Mo1, Zw2}. In original work of Lovejoy and Osburn \cite{LO, LO2}, the results of Hickerson and the author on relating Hecke-type double sums to Appell--Lerch sums \cite{HM} were instrumental in determining mock theta behaviour of multisum $q$-hypergeometric series.  One finds traces of Appell--Lerch sums throughout the Lost Notebook \cite{RLN}, where many identities express Eulerian series in terms of what are essentially $m(x,q,z)$ functions.  For Ramanujan's sixth order mock theta functions $\phi(q)$ and $\sigma(q)$,  one finds slightly rewritten \cite{AH, RLN}:
\begin{align*}
\phi(q)&:=\sum_{n= 0}^{\infty}\frac{(-1)^nq^{n^2}(q;q^2)_n}{(-q)_{2n}} =2m(q,q^3,-1), \ \sigma(q):=\sum_{n= 0}^{\infty}\frac{q^{\binom{n+2}{2}}(-q)_n}{(q;q^2)_{n+1}}=-m(q^2,q^6,q).
\end{align*}

 Appell--Lerch sums satisfy several well-known functional equations and identities, which we collect in the form of a proposition, see for example \cite{HM}.  Here, the term generic means that the variables do not cause singularities in the Appell--Lerch sums or in the quotients of the theta functions.
 \begin{proposition}  For generic $x,z,z_0,z_1\in \mathbb{C}^*$ 
 {\allowdisplaybreaks \begin{subequations}
\begin{equation}
m(x,q,z)=m(x,q,qz),\label{equation:m-fnq-z}
\end{equation}
\begin{equation}
m(qx,q,z)=1-xm(x,q,z),\label{equation:m-fnq-x}
\end{equation}
\begin{equation}
m(x,q,z_1)-m(x,q,z_0)=\frac{z_0J_1^3j(z_1/z_0;q)j(xz_0z_1;q)}{j(z_0;q)j(z_1;q)j(xz_0;q)j(xz_1;q)}.\label{equation:m-change-z}
\end{equation}
\end{subequations}}
\end{proposition}
\noindent Although one does not find anything as explicit as (\ref{equation:m-fnq-z})--(\ref{equation:m-change-z}) in \cite{RLN}, one does find many specializations of the identities.  For example, (\ref{equation:m-change-z}) specializes to the following Lost Notebook relation  for the above sixth orders  \cite[$(0.19)_R$]{AH}:
\begin{equation}
\phi(q^2)+2\sigma(q)=\prod_{n\ge 1}(1+q^{2n-1})^2(1-q^{6n})(1+q^{6n-3})^2.\label{equation:6th-id}
\end{equation}
Another example is  \cite[Entry $12.4.1$]{ABI} which is a combination of (\ref{equation:m-fnq-flip}) and (\ref{equation:m-change-z}).  One can also view (\ref{equation:6th-id}) as a linear combination of Eulerian series which essentially yields a weight $1/2$ weakly holomorphic modular form.  

Ramanujan also expanded more involved Eulerian series in terms of Appell--Lerch like sums.  We recall \cite[Proposition $2.6$]{Mo1}:
\begin{align}
\sum_{n= 0}^{\infty}\frac{(-1)^nq^{n^2}(q;q^2)_{n}}{(x;q^2)_{n+1}(q^2/x;q^2)_{n}}&=m(-x,q,-1)+\frac{J_{1,2}^2}{2j(x;q)},\label{equation:RLNid2}\\
\Big ( 1-\frac{1}{x}\Big )\sum_{n=0}^{\infty}\frac{(-1)^n(q;q^2)_nq^{(n+1)^2}}{(xq;q^2)_{n+1}(q/x;q^2)_{n+1}}&=m(-x,q,-1)-\frac{J_{1,2}^2}{2j(x;q)},\label{equation:RLNid4}
\end{align}
where both are rewritten equations of \cite{RLN}  proved in Andrews \cite{AM}.  Identities such as (\ref{equation:RLNid2})--(\ref{equation:RLNid4}) and the techniques of \cite{HM} are useful in finding additional $q$-hypergeometric and bilateral $q$-hypergeometric series with (mixed) mock modular behaviour \cite{Mo1}.  

In Ramanujan's last letter to Hardy, he included mock theta functions of orders three, five, and seven.  The third orders could each be written as a special case of the generalised Lambert series $g(x,q).$ For example, take the third order $f(q)$:
\begin{equation*}
f(q):=\sum_{n= 0}^{\infty}\frac{q^{n^2}}{(-q)_n^2}=2-2g(-1,q),
\end{equation*}
where, see \cite[Proposition $3.2$] {HM}:
\begin{align*}
g(x,q)&:=x^{-1}\Big ( -1 +\sum_{n=0}^{\infty}\frac{q^{n^2}}{(x)_{n+1}(q/x)_{n}} \Big )=\sum_{n=0}^{\infty}\frac{q^{n(n+1)}}{(x)_{n+1}(q/x)_{n+1}}\\
&\ \ =-x^{-1}m(q^2x^{-3},q^3,x^2)-x^{-2}m(qx^{-3},q^3,x^2).
\end{align*}
Not until the discovery of the Lost Notebook and subsequent work of Andrews, Garvan, and Hickerson \cite{AG, H5, H7} was it realized that the fifth and seventh orders could each be expressed as the sum of a $g(x,q)$ and a single quotient of theta functions.  These expressions for the fifth orders were the so-called mock theta conjectures.  For the fifth order $f_0(q)$:
\begin{equation}
f_0(q):=\sum_{n= 0}^{\infty}\frac{q^{n^2}}{(-q)_n}=-2q^2g(q^2,q^{10})+\frac{J_{5,10}J_{2,5}}{J_1}.\label{equation:f0-conj}
\end{equation}

In \cite{HM}, Hickerson and the author developed and refined the notion of expanding Ramanujan's classical mock theta functions in terms of building blocks.   We showed that if one allows repetition in $x$  in (\ref{equation:mxqz-def}), one can always write these functions entirely in terms of $m(x,q,z)$'s.  If one does not allow for duplication in $x$, one can adjust the $z$'s such that there is only a single quotient of theta functions.  For $f_0(q)$ \cite{HM}:
\begin{align}
f_0(q)&=m(q^{14},q^{30},q^{14})+m(q^{14},q^{30},q^{29})+q^{-2}m(q^{4},q^{30},q^4)+q^{-2}m(q^{4},q^{30},q^{19}) \notag\\
&=2m(q^{14},q^{30},q^4)+2q^{-2}m(q^{4},q^{30},q^4)+\frac{J_{5,10}J_{2,5}}{J_1}.\label{equation:mf0-expansion}
\end{align}
Such expansions are of interest when studying the partial theta function duals \cite{Mo1}.  To prove such expressions, we introduced the following Appell--Lerch sum identity.

\begin{theorem} \label{theorem:msplit-general-n} \cite{HM} For generic $x,z,z'\in \mathbb{C}^*$
{\allowdisplaybreaks \begin{align*}
m(&x,q,z) = \sum_{r=0}^{n-1} q^{{-\binom{r+1}{2}}} (-x)^r m\big(-q^{{\binom{n}{2}-nr}} (-x)^n, q^{n^2}, z' \big)\notag\\
& + \frac{z' J_n^3}{j(xz;q) j(z';q^{n^2})}  \sum_{r=0}^{n-1}
\frac{q^{{\binom{r}{2}}} (-xz)^r
j\big(-q^{{\binom{n}{2}+r}} (-x)^n z z';q^n\big)
j(q^{nr} z^n/z';q^{n^2})}
{ j\big(-q^{{\binom{n}{2}}} (-x)^n z', q^r z;q^n\big )}.
\end{align*}}
\end{theorem}

Independently, Gordon and McIntosh \cite{GM} expanded mock thetas in terms of multiple building blocks, but  not in a comprehensive manner like (\ref{equation:mf0-expansion}).  Also, no result like Theorem \ref{theorem:msplit-general-n} was obtained.  In fact, Theorem \ref{theorem:msplit-general-n} and the $m(x,q,z)$ expansions in \cite{HM}  enabled Lovejoy and Osburn \cite{LO1} to give a short proof of conjectured  identities for the tenth orders  \cite{GM}.  They also note that since all classical mock theta functions can be written in terms of Appell-Lerch sums (see \cite{HM}), one can easily prove similar identities for $2$nd, $3$rd, $6$th, and $8$th orders, see \cite[$(5.2)$, $(3.12)$, $(5.10)$]{GM} and the top of page $125$ in \cite{GM}.

\section{A few more technical details}
We list a few more technical details \cite{HM}.  Some useful theta function identities are
\begin{gather}
\overline{J}_{0,1}=2\overline{J}_{1,4}=\frac{2J_2^2}{J_1},  \overline{J}_{1,2}=\frac{J_2^5}{J_1^2J_4^2},   J_{1,2}=\frac{J_1^2}{J_2},   \overline{J}_{1,3}=\frac{J_2J_3^2}{J_1J_6}, 
 J_{1,4}=\frac{J_1J_4}{J_2}.
\end{gather}
We state additional general identities for the theta function:
\begin{subequations}
\begin{equation}
j(q^n x;q)=(-1)^nq^{-\binom{n}{2}}x^{-n}j(x;q), \ \ n\in\mathbb{Z},\label{equation:1.8}
\end{equation}
\begin{equation}
j(x;q)=j(q/x;q)=-xj(x^{-1};q)\label{equation:1.7},
\end{equation}
\begin{equation}
j(x;q)={J_1}j(x,qx;q^2)/{J_2^2} \label{equation:1.10}
\end{equation}
\begin{equation}
j(z;q)=j(-qz^2;q^4)-zj(-q^3z^2;q^4),\label{equation:jsplit}
\end{equation}
\begin{equation}
j(x^2;q^2)=j(x;q)j(-x;q)/{J_{1,2}},\label{equation:1.12}
\end{equation}
\end{subequations}
We also recall the reciprocal of Jacobi's theta product:
\begin{equation}
\sum_{n=-\infty}^{\infty}\frac{(-1)^nq^{\binom{n+1}{2}}}{1-q^nz}=\frac{J_1^3}{j(z;q)}.\label{equation:RJTP}
\end{equation}
Finally we note for generic $x,y,z\in \mathbb{C}^*$:
\begin{equation}
j(x;q)j(y;q)=j(-xy;q^2)j(-qx^{-1}y;q^2)-xj(-qxy;q^2)j(-x^{-1}y;q^2).\label{equation:H1Thm1.1}
\end{equation}

\section{The theorem and the alternate proof}

Motivated by Dyson's rank differences,  Bringmann, Ono, and Rhoades \cite{BrOR} used the theory of harmonic weak Maass forms in order to identify linear combinations of Eulerian series which are weakly holomorphic modular forms.   

We recall the relevant notation from \cite{BrOR}.   Define $K^{\prime} (w;z)$, $K^{\prime\prime}(w;z)$, $H^{\prime} (a,c,w;z)$, by
{\allowdisplaybreaks \begin{align}
K^{\prime} (\omega;z)&:=\sum_{n=0}^{\infty}\frac{(-1)^nq^{n^2}(q;q^2)_n}{(\omega q^2;q^2)_{n}(\omega^{-1} q^2;q^2)_{n}},\\
K^{\prime\prime} (\omega;z)&:=\sum_{n=1}^{\infty}\frac{(-1)^nq^{n^2}(q;q^2)_{n-1}}{(\omega q;q^2)_{n}(\omega^{-1} q;q^2)_{n}},\\
H^{\prime} (a,c,w;z)&:=\sum_{n=0}^{\infty}\frac{q^{\frac{1}{2}n(n+1)}(-q)_n}{(\omega q^{\frac{a}{c}})_{n+1}(\omega q^{1-\frac{a}{c}})_{n+1}},
\end{align}
where $q:=e^{2\pi i z}$ and $0< a < c.$  Further, let $\zeta_c:=e^{2\pi i /c}$ and $f_c:=2c/\gcd(c,4).$  Let
\begin{align}
\tilde{K}(a,c;z):=\frac{1}{4}\textup{csc} \Big ( \pi  \frac{a}{c} \Big )q^{-\frac{1}{8}}K^{\prime}(\zeta_c^a;z)+\textup{sin}\Big(  \pi \frac{a}{c}\Big ) q^{-\frac{1}{8}}K^{\prime\prime}(\zeta_c^a;z),\\
\tilde{H}(a,c;z):=q^{\frac{a}{c}(1-\frac{a}{c})}(H^{\prime}(a,c,1;z)-H^{\prime}(a,c,-1;z)),\label{equation:Htilde}
\end{align}}%
where a sign error has been corrected in (\ref{equation:Htilde}).   One of the main results of \cite{BrOR} reads
\begin{theorem}{\cite[Theorem $1.3$]{BrOR}}\label{theorem:thmBrOR} Let $0<a<c.$   In the notation above, $\tilde{H}(a,c;4f_c^2z)$ is a weight $1/2$ weakly holomorphic modular form on $\Gamma_1(64f_c^4)$ and $\tilde{K}(a,c;2f_c^2z)$ is a weight $1/2$ weakly holomorphic modular form on $\Gamma_1(64f_c^4).$
\end{theorem}
\noindent  The proof of Theorem \ref{theorem:thmBrOR} is lengthy and detailed and yields no explicit formulas for $\tilde{H}$ and $\tilde{K}$.  Here, we change the context to Appell--Lerch sums and  employ the techniques of \cite{HM} to shorten the proof of Theorem \ref{theorem:thmBrOR} and discuss connections to recent work of Kang \cite{K}.

 \begin{theorem}\label{theorem:new}  Let $0<a<c.$   In the notation above, we have
\begin{equation*}
\tilde{H}(a,c;z)=2q^{\frac{a}{c}(1-\frac{a}{c})}\frac{J_2^3}{J_{1,2}j(q^{\frac{2a}{c}};q^2)}, \ \ \tilde{K}(a,c;z)=-\frac{i\zeta_{c}^{a/2}q^{-\frac{1}{8}}}{2}\frac{J_{1,2}^2}{j(\zeta_c^a;q)}.
\end{equation*}
\end{theorem}

For Theorem \ref{theorem:new}, we give two proofs of the explicit expressions for $\tilde{K}$.  The first uses (\ref{equation:RLNid2}) and (\ref{equation:RLNid4}) while the second demonstrates how to use new Appell-Lerch sum properties to go from the Watson-Whipple results of Kang \cite{K} to identities (\ref{equation:RLNid2}) and (\ref{equation:RLNid4}).

\begin{proof} [Proof of Theorem \ref{theorem:new}]
Using (\ref{equation:RLNid2}) and (\ref{equation:RLNid4}), we see
\begin{align*}
K^{\prime} (\omega;z)&=(1-\omega)\Big ( m(-\omega,q,-1)+\frac{J_{1,2}^2}{2j(\omega;q)}\Big ),\\
K^{\prime\prime} (\omega;z)&=\frac{\omega}{1-\omega}\Big (  m(-\omega,q,-1)-\frac{J_{1,2}^2}{2j(\omega;q)} \Big ).
\end{align*}
The explicit form for $\tilde{K}$ is then immediate upon writing sine and cosecant to exponential  form.   This completes the first proof.  

For the second proof, we first note that Kang \cite{K} showed via Watson-Whipple \cite[eq. $(2.5.1)$, p. 43]{GR}:
\begin{align}
\frac{1}{1-\omega}K^{\prime}(\omega;\tau)&=\frac{1}{\overline{J}_{1,4}}\sum_{n=-\infty}^{\infty}\frac{q^{2n^2+n}}{1-\omega q^{2n}},\label{equation:kang1}\\
\Big( 1-\frac{1}{\omega}\Big )K^{\prime \prime}(\omega;\tau)&=-\frac{1}{\overline{J}_{1,4}}\sum_{n=-\infty}^{\infty}\frac{q^{2n^2+3n+1}}{1-\omega q^{2n+1}}\label{equation:kang2}.
\end{align}
In \cite{K}, one adds (\ref{equation:kang1}) and (\ref{equation:kang2}) and uses (\ref{equation:RJTP}).  Identity (\ref{equation:kang1}) is also in \cite{RLN,AM}.   We show how to use Theorem \ref{theorem:msplit-general-n} and elementary theta function properties to obtain (\ref{equation:RLNid2}) from (\ref{equation:kang1}).  The proof for (\ref{equation:RLNid4}) is similar.  We begin with
{\allowdisplaybreaks \begin{align*}
\frac{1}{\overline{J}_{1,4}}\sum_{n=-\infty}^{\infty}\frac{q^{2n^2+n}}{1-\omega q^{2n}}
&=\frac{1}{\overline{J}_{1,4}}\sum_{n=-\infty}^{\infty}\frac{q^{2n^2+n}}{1-\omega^2 q^{4n}}+\frac{\omega}{\overline{J}_{1,4}}\sum_{n=-\infty}^{\infty}\frac{q^{2n^2+3n}}{1-\omega^2 q^{4n}}\\
&=m(-\omega^2q,q^4,-q^3)+\omega q^{-1}m(-\omega^2q^{-1},q^4,-q)\\
&=m(-\omega^2q,q^4,-q)+\omega q^{-1}m(-\omega^2q^{-1},q^4,-q) +\frac{J_{1,2}^2j(-q\omega^2;q^4)}{j(\omega;q)j(-\omega;q)},
\end{align*}}%
where we have used (\ref{equation:1.8}), (\ref{equation:1.7}), (\ref{equation:1.10}), (\ref{equation:1.12}), (\ref{equation:m-fnq-z}), (\ref{equation:m-change-z}) and simplified.  Using Theorem \ref{theorem:msplit-general-n} with $n=2$, $x=-\omega$, $z=-1$, and $z^{\prime}=-q$ and more simplifying we obtain
{\allowdisplaybreaks \begin{align*}
\frac{1}{\overline{J}_{1,4}}\sum_{n=-\infty}^{\infty}\frac{q^{2n^2+n}}{1-\omega q^{2n}}
&=m(-\omega,q,-1) +\frac{J_{1,2}^2j(-q\omega^2;q^4)}{j(\omega;q)j(-\omega;q)}\\
&\ \ \ \ \ -\frac{J_{1,2}^2}{2j(\omega;q)^2j(-\omega;q)}\Big [ j(-\omega^2;q^2)\overline{J}_{1,2}-\omega j(-q\omega^2;q^2)\overline{J}_{0,2}\Big ]\\
&=m(-\omega,q,-1) +\frac{J_{1,2}^2j(-q\omega^2;q^4)}{j(\omega;q)j(-\omega;q)}-\frac{J_{1,2}^2}{2j(-\omega;q)}\\
&=m(-\omega,q,-1) +\frac{J_{1,2}^2}{2j(-\omega;q)j(\omega;q)}\Big [ 2j(-q\omega^2;q^4)-j(\omega;q)\Big ]\\
&=m(-\omega,q,-1) +\frac{J_{1,2}^2}{2j(-\omega;q)j(\omega;q)}\Big [ j(-q\omega^2;q^4)+\omega j(-q^3\omega^2;q^4)\Big ]\\
&=m(-\omega,q,-1) +\frac{J_{1,2}^2}{2j(\omega;q)},
\end{align*}}%
where the second equality follows from (\ref{equation:H1Thm1.1}) and the last two from (\ref{equation:jsplit}).

For $\tilde {H}$, we begin with \cite[$(4.1)$]{BrOR}:
\begin{align*}
H(a,b,c;z):&=\frac{1}{J_{1,2}}\sum_{n=-\infty}^{\infty}\frac{(-1)^nq^{n+\frac{a}{c}}q^{n(n+1)}}{1-\zeta_c^bq^{n+\frac{a}{c}}}\\
&=\frac{q^{\frac{a}{c}}}{J_{1,2}}\sum_{n=-\infty}^{\infty}\frac{(-1)^nq^{n^2+2n}}{1-\zeta_c^{2b}q^{\frac{2a}{c}}q^{2n}}+\frac{\zeta_c^{b}q^{\frac{2a}{c}}}{J_{1,2}}\sum_{n=-\infty}^{\infty}\frac{(-1)^nq^{n^2+3n}}{1-\zeta_c^{2b}q^{\frac{2a}{c}}q^{2n}}\\
&=-q^{\frac{a}{c}-1}m(\zeta_c^{2b}q^{\frac{2a}{c}-1},q^2,q)+\frac{\zeta_c^{-b}J_2^3}{J_{1,2}j(\zeta_c^{2b}q^{\frac{2a}{c}};q^2)},
\end{align*}
where in the last line we used (\ref{equation:mxqz-def}) and (\ref{equation:RJTP}).  The result follows upon recalling \cite[$(4.5)$]{BrOR}
\begin{equation*}
\tilde{H}(a,c;z)=q^{\frac{a}{c}(1-\frac{a}{c})}(H(a,0,c;z)-H(a,c/2,c;z)).\qedhere
\end{equation*}
\end{proof}

\section*{Acknowledgements} We would like to thank Wadim Zudilin and the referee for helpful comments and suggestions.

\end{document}